\documentclass[11pt]{amsart}
\usepackage{amssymb,amsfonts,amsthm}
\usepackage{amsmath}
\usepackage[makeroom]{cancel}
\usepackage{mathtools}
\usepackage{mathrsfs,pifont}
\usepackage{slashed,mathabx} 
\usepackage[bbgreekl]{mathbbol}
\usepackage{tikz}
\usepackage{framed}
\usetikzlibrary{calc}
\usetikzlibrary{arrows.meta,positioning}

\usepackage[font={small,sl}]{caption}
\usepackage[all]{xy}

\usepackage{hyperref}

\newtheorem{theorem}{Theorem}[section]
\newtheorem{lemma}[theorem]{Lemma}
\newtheorem{corollary}[theorem]{Corollary}
\theoremstyle{definition}

\theoremstyle{remark}
\newtheorem{remark}[theorem]{Remark}
\numberwithin{equation}{section}

\calclayout

\makeatletter
\def\blfootnote{\xdef\@thefnmark{}\@footnotetext}
\makeatother

\newcommand{\Sn}{\mathfrak{S}_n}

\begin{document}

\title{Spectral Gap of Biased Adjacent-Transposition Chains}

\author{Gary Greaves}
\address{Division of Mathematical Sciences, Nanyang Technological University, 21 Nanyang Link, Singapore 637371}
\email{gary@ntu.edu.sg}

\author{Haoran Zhu}
\address{Division of Mathematical Sciences, Nanyang Technological University, 21 Nanyang Link, Singapore 637371}
\email{zhuh0031@e.ntu.edu.sg}

\date{}
\maketitle

\begin{abstract}
We establish a sharp lower bound on the spectral gap of the biased adjacent-transposition Markov chain on the symmetric group. 
As a consequence, we resolve a longstanding conjecture of Fill, proving that among all regular probability vectors, the minimum spectral gap of the transition matrix is attained by the uniform probability vector. 
We also characterise the regular probability vectors attaining the minimum spectral gap and determine the exact multiplicity of the corresponding second-largest eigenvalue.
Our proof relies on a novel algebraic decomposition of the transition matrix into elementary orthogonal projections.
\end{abstract}

{\noindent \it Keywords:} extremal spectral gap, adjacent transposition, orthogonal projection, Markov transition matrix, eigenvalue multiplicity.

{\noindent \it Mathematics Subject Classification (2020):} 60J10; 05E18.

\section{Introduction}

Random walks generated by transpositions on the symmetric group form one of the most natural and well-studied classes of Markov chains on permutations~\cite{AD,AF,BBHM,BMRS,D1,D2}.  
In the uniform case, their spectral theory is classical and closely connected with the representation theory of the symmetric group $\mathfrak{S}_n$~\cite{AF,D1,DSC}.  
The introduction of bias changes the picture substantially: the stationary measure becomes non-uniform, the algebraic symmetry of the chain is broken, and the spectral analysis becomes markedly more delicate. 
Despite the apparent simplicity of these models, a comprehensive understanding of their spectral behaviour and mixing times has remained elusive, with exact results largely confined to special cases \cite{BMRS,GLV,HW,MS,MSS}.
 
In 2003, Fill~\cite{Fill} conjectured 
that among all regular biased adjacent-transposition chains, the spectral gap is minimised by the uniform chain.  
The problem of proving or disproving this conjecture, now known as \emph{Fill's Gap Problem}, identifies a natural extremal problem at the interface of probability, combinatorics, and the spectral theory of permutation chains.  
In this paper, we resolve Fill's conjecture, prove a sharp lower bound on the spectral gap for arbitrary bias, and determine the full structure of the extremal regular chains.

Fill's problem may be viewed as complementary to Aldous's spectral gap conjecture, proved by Caputo, Liggett and Richthammer~\cite{CLR}.  
Whereas Aldous relates the spectral gap of a symmetric transposition chain to that of the associated one-particle random walk, Fill asks for the extremal behaviour of the spectral gap within a family of biased adjacent-transposition chains.  
The two problems are thus close in spirit but quite different in mechanism: in Fill's setting, the local bias destroys the symmetry available in the uniform case, and a different method is required.

\subsection{Biased adjacent-transposition chains}

Let $n\ge 2$, denote by $\Sn$ the symmetric group on $[n]:=\{1,2,\dots,n\}$, and write a permutation $x \in \mathfrak S_n$ of $[n]$ in one-line notation as
$x=(x_1,\dots,x_n)$. 
We consider the adjacent-transposition Markov chain $\mathcal M$ on $\mathfrak S_n$ weighted by given probabilities $p_{i,j}$ for all $i < j$, where $0 < p_{i,j} < 1$.
The probabilities $p_{j,i}$ are determined by the equation $p_{j,i} = 1 - p_{i,j}$.
These probabilities are collected together in the \emph{probability vector} $\mathbf p =(p_{i,j})_{1 \le i \ne j \le n}$.
From the current state $x$, choose
$r\in\{1,\dots,n-1\}$ uniformly. 
If the consecutive labels at positions $r$ and $r+1$ are $x_r$ and
$x_{r+1}$, respectively, then place $x_{r}$ immediately to the left of $x_{r+1}$ with probability $p_{x_{r},x_{r+1}}$ and place $x_{r+1}$ immediately to the left of $x_{r}$ with probability $p_{x_{r+1},x_{r}}$. 
The \emph{transition matrix} $K$ of the Markov chain $\mathcal M$ is defined as follows. 
For each $x \in \Sn$ and $r \in \{1,\dots,n-1\}$, if $y$ is the permutation obtained from $x$ by swapping the entries in positions $r$ and $r+1$ then
\begin{equation*}
K_{x,y}=\frac{p_{x_{r+1},x_r}}{n-1},
\qquad
K_{x,x}=\frac{1}{n-1}\sum_{r=1}^{n-1}p_{x_r,x_{r+1}},
\end{equation*}
and all other transition probabilities are zero.
For example, when $n = 3$, ordering the elements of $\mathfrak S_3$ as $(1,2,3),(2,1,3),(1,3,2),(2,3,1),(3,1,2),(3,2,1)$, the transition matrix $K$ is given by

\noindent\scalebox{0.94}{$\displaystyle 
K = \frac 1 2 \begin{bmatrix} 
p_{1,2}+p_{2,3} & p_{2,1} & p_{3,2} & 0 & 0 & 0 \\
p_{1,2} & p_{2,1}+p_{1,3} & 0& p_{3,1} & 0  & 0 \\
p_{2,3} & 0 & p_{1,3}+p_{3,2} & 0 & p_{3,1} & 0 \\
0 & p_{1,3} & 0  & p_{2,3}+p_{3,1} & 0 & p_{3,2} \\
0  & 0 & p_{1,3} & 0 & p_{3,1}+p_{1,2} & p_{2,1} \\
0 & 0 & 0 & p_{2,3} & p_{1,2} & p_{3,2}+p_{2,1} 
\end{bmatrix}.
$}

The probability vector $\mathbf p$ is called \emph{regular} if
\begin{align}
 p_{i-1,i}&\ge \tfrac 1 2, && 2\le i\le n, \label{eq:regular1}\\
 p_{i-1,j}&\ge p_{i,j}, && 2\le i<j\le n, \label{eq:regular2}\\
 p_{i,j+1}&\ge p_{i,j}, && 1\le i<j\le n-1. \label{eq:regular3}
\end{align}
We refer to the (regular) case when $p_{i,j} = 1/2$ for all $i \ne j$ as the \emph{uniform case}.
The \emph{spectral gap} $\lambda_K$ of $K$ is defined as the difference of the two largest eigenvalues of $K$, that is, $\lambda_K:=1-\beta_K$, where $\beta_K$ is the second-largest eigenvalue of $K$.



\subsection{Main results}

Our main theorem gives an explicit lower bound on the spectral gap for an arbitrary probability vector.  Given a probability vector $\mathbf p$ of order $n$, define
\[
m_{\mathbf p}:=\max_{1\le i<j<k\le n}\sqrt{p_{i,j}p_{j,k}p_{k,i}+p_{k,j}p_{j,i}p_{i,k}}.
\]

\begin{theorem}\label{thm:main}
For $n\ge 3$, we have
\[
\lambda_K\ge \frac{1-2m_{\mathbf p}\cos(\tfrac{\pi}{n})}{n-1}.
\]
\end{theorem}

In the regular case, $m_{\mathbf p}$ is bounded above by $1/2$; see~\cite[Evidence in favor of the conjecture (b)]{Fill}.  
Hence, Theorem~\ref{thm:main} immediately yields the following corollary.

\begin{corollary}\label{cor:reg}
If $\mathbf p$ is regular, then $\lambda_K\ge (1-\cos(\tfrac{\pi}{n}))/(n-1)$.
The case of equality implies $m_{\mathbf p}=1/2$.
\end{corollary}

Wilson~\cite{W} proved
that $\lambda_K=(1-\cos(\tfrac{\pi}{n}))/(n-1)$ in the uniform case.
This value was already implicit in earlier work of Bacher~\cite{B}, since in the uniform case the transition matrix is an affine transformation of the adjacency matrix of the permutahedron, whose second-largest eigenvalue he determined.
Thus, Corollary~\ref{cor:reg} is sharp and hence verifies Fill's conjecture.

\begin{corollary}[Fill's spectral gap conjecture]\label{cor:fill}
Among all regular probability vectors, the spectral gap of the transition matrix attains its minimum in the uniform case.
\end{corollary}

Among all regular probability vectors, the minimum spectral gap is attained not only in the uniform case, but by a specific family of probability vectors, which we characterise in the next theorem.

\begin{theorem}\label{thm:Fillc1}
Let $n\ge 3$ and let $\mathbf p$ be a regular probability vector.  
Then
\[
\lambda_K=\frac{1-\cos(\tfrac{\pi}{n})}{n-1}
\]
if and only if there exists $i\in[n]$ such that $p_{i,j}=1/2$ for all $j\ne i$.
\end{theorem}

Fill~\cite[Stronger conjectures (c)]{Fill} also conjectured a precise formula for the multiplicity of the second-largest eigenvalue in the extremal case.
Our next result proves this stronger conjecture.

Denote by $\nu_{\mathbf p}$ the number of indices $i\in[n]$ for which $p_{i,j}=1/2$ for all $j\ne i$, and denote by $\mathrm{mult}_K(\theta)$ the multiplicity of $\theta$ as an eigenvalue of $K$.  

\begin{theorem}\label{thm:Fillc2}
Let $n\ge 3$ and let $\mathbf p$ be a regular probability vector.  
Suppose that $\lambda_K=(1-\cos(\tfrac{\pi}{n}))/(n-1)$.
Then
\[
\operatorname{mult}_K(1-\lambda_K)=
\begin{cases}
\nu_{\mathbf p}, & \nu_{\mathbf p}\notin\{n,n-2\},\\
n-1, & \nu_{\mathbf p}\in\{n,n-2\}.
\end{cases}
\]
\end{theorem}

Together, Theorem~\ref{thm:Fillc1} and Theorem~\ref{thm:Fillc2} give a precise description of the regular adjacent-transposition Markov chains that are extremal for the spectral gap.  
In particular, among such extremal chains, the multiplicity of the second-largest eigenvalue measures the extent to which the chain retains features of the uniform case. 
Furthermore, Theorem~\ref{thm:Fillc2} sharpens both \cite[Théorème]{B} and \cite[Theorem 14]{W}, which only gave a lower bound on the multiplicity of the eigenvalue $1-\lambda_K$ in the uniform case.

\begin{remark}
\label{rem:permutahedron}
    The Cayley graph $\Gamma$ on $\mathfrak S_n$ generated by adjacent transpositions is also known as the permutahedron.
    By Theorem~\ref{thm:Fillc2}, the multiplicity of the second-largest eigenvalue of $\Gamma$ is $n-1$.
    On the other hand, this eigenvalue already occurs on the standard representation of $\mathfrak S_n$, whose dimension is $n-1$~\cite{B}.  
It follows that the entire second eigenspace of $\Gamma$ is afforded by the standard representation. 
In particular, no other irreducible representation of $\mathfrak S_n$ attains the second-largest eigenvalue of the permutahedron.
    The problem of determining which representations of $\mathfrak S_n$ attain the second-largest eigenvalue of Cayley graphs $\Gamma$ is a well-studied problem and the cases non-normal Cayley graphs are notoriously difficult~\cite{YL}.
\end{remark}


\subsection{Proof strategy}
\label{sec:proof-strategy}

The central structural ingredient in our proof is an algebraic decomposition of the transition matrix. 
A key result of Fill~\cite[Theorem~2]{Fill} establishes that the matrices $K$ and $I-K$ are similar, reducing the analysis of the spectral gap $\lambda_K$ to bounding the smallest positive eigenvalue of $K$.
To achieve this, we decompose the transition matrix as
$$K=\frac{1}{n-1}\sum_{r=1}^{n-1}\mathsf{E}_r,$$
where each $\mathsf{E}_r$ is an \emph{elementary transition matrix} acting strictly on the adjacent pair of positions $(r,r+1)$. 
We demonstrate that under a natural weighted inner product on $\mathbb{R}^{\mathfrak{S}_n}$ governed by the stationary measure of $\mathcal M$, each $\mathsf{E}_r$ acts as an orthogonal projection. 
This geometric viewpoint allows us to extract a quantitative lower bound on the angle between the vectors $\mathsf{E}_r f$ and $\mathsf{E}_{r+1}f$ from the smallest positive eigenvalue of the local sum $\mathsf{E}_r+\mathsf{E}_{r+1}$. 
Since elementary transition matrices with indices differing by at least two commute, these local bounds can be assembled to produce a global lower bound from the smallest positive eigenvalue of $K$, resulting in the proof of Theorem~\ref{thm:main}.

The same framework also controls the case of equality.  
The global equality for the smallest positive eigenvalue of $K$ forces a sequence of local relations among the elementary transition matrices, and these relations are strong enough to recover the precise structure asserted in Theorem~\ref{thm:Fillc1}.  
A further analysis of the corresponding eigenspaces then leads to Theorem~\ref{thm:Fillc2}.
Thus, the projection method not only bounds the spectral gap, but also characterises the extremal probability vectors.

\subsection{Consequences and outlook}

Theorem~\ref{thm:Fillc1} shows that the lower bound in Corollary~\ref{cor:reg} is attained only on a subclass of regular probability vectors, namely those for which $\nu_{\mathbf p}\ge 1$.  
Consequently, if $\nu_{\mathbf p}=0$, then we obtain a strict lower bound on the spectral gap of $K$,
\[
\lambda_K>\frac{1-\cos(\pi/n)}{n-1}.
\]

By Corollary~\ref{cor:reg}, if $\mathbf p$ is a regular probability vector then 
\[
\lambda_K \ge \frac{1-\cos(\pi/n)}{n-1}\ge \frac 2 {n^3}.
\]
It follows at once that $1/\lambda_K=O(n^3)$
for every regular probability vector.  
This conclusion substantially enlarges the range of biased adjacent-transposition chains for which polynomial relaxation is known.  
Earlier polynomial mixing or relaxation bounds were established for special subclasses of biased permutation chains~\cite{BMRS,HW,MS,MSS}, and more recently for the general adjacent-transposition chain under the stronger assumption that $p_{ij}>1/2+\varepsilon$ for all $i<j$~\cite{GLV}. 

Biased adjacent-transposition chains arise in models of self-organising lists~\cite{H,HH}, card shuffling~\cite{LL}, online algorithms~\cite{C}, ranking via the Bradley--Terry model~\cite{BT}, and approximation of Vandermonde permanents~\cite{BSV}; see also~\cite{Fill}.

As shown in Theorem~\ref{thm:Fillc2} and Remark~\ref{rem:permutahedron}, in addition to dealing with bias, our techniques allow us to strengthen results in the uniform case of Bacher~\cite{B} and Wilson~\cite{W} on the multiplicity of the second-largest eigenvalue of the Cayley graph $\Gamma$ on $\mathfrak S_n$ generated by adjacent transpositions.
This suggests that the projection framework developed here may be useful in other spectral problems for transposition-generated chains on $\mathfrak S_n$.

\subsection{Acknowledgement}

We are grateful to James Allen Fill for his insightful comments on an earlier draft of this paper.

\section{Transitions and projections}

We will fix $n$ and the probability vector $\mathbf p$ throughout.
Consider the space $\mathbb{R}^{\mathfrak{S}_n}$ of real vectors with coordinates indexed by permutations.
We will move interchangeably between the language of matrices and operators on $\mathbb{R}^{\mathfrak{S}_n}$.
To ease the notation, we denote the $x$-coordinate of a vector $f \in \mathbb{R}^{\mathfrak{S}_n}$ by $f(x)$.

As observed by Fill \cite{Fill}, the Markov chain $\mathcal M$ is reversible with respect to the probability measure
\begin{equation*} 
\mu(x)=Z^{-1}\prod_{1\le u<v\le n}p_{x_u,x_v},
\end{equation*}
where $Z$ is the normalising constant. 
Equip $\mathbb{R}^{\mathfrak{S}_n}$ with the weighted inner product $\langle \cdot, \cdot \rangle$ and induced norm $\|\cdot \|$ defined by
\[
\langle f, {g} \rangle
:=  \sum_{x \in \mathfrak{S}_n} {f}(x)\, {g}(x)\, \mu(x),
\qquad
\|f\| := \sqrt{\langle f, {f} \rangle}.
\]
We choose this inner product so that the elementary transition matrices defined below act as orthogonal projections.

\subsection{Elementary transition matrices}

Let $r \in \{1, \dots, n-1\}$.
For $x = (x_1, \dots, x_n) \in \mathfrak{S}_n$, write
\[
x^{\tau_r} := (x_1, \dots, x_{r-1}, x_{r+1}, x_r, x_{r+2}, \dots, x_n)
\]
for the permutation obtained by swapping the entries in positions $r$ and $r+1$.

Define the matrix $\mathsf E_r$ by
\[
(\mathsf E_r)_{x,y} = \begin{cases}
    p_{x_r, x_{r+1}}, & \text{if } y = x; \\
    p_{x_{r+1}, x_r}, & \text{if } y = x^{\tau_r}; \\
    0, & \text{otherwise.}
\end{cases}
\]
We call $\mathsf E_r$ the $r$-th \emph{elementary transition matrix}.
Note that 
\begin{equation}
\label{eqn:Eraction}
    \mathsf E_r f(x) = p_{x_r, x_{r+1}} f(x) + p_{x_{r+1}, x_r} f(x^{\tau_r}).
\end{equation}

Our first result shows that elementary transition matrices commute when their indices differ by at least $2$.

\begin{lemma}\label{lem:far-commute}
Let $r,s\in \{1,\dots, n-1\}$.
Suppose $|r-s| \ge 2$.
Then $\mathsf E_r\mathsf E_s=\mathsf E_s\mathsf E_r$. 
\end{lemma}

For the sake of completeness, we provide a short proof.

\begin{proof}
For each $f \in \mathbb R^{\Sn}$ and $x\in \mathfrak S_n$, applying \eqref{eqn:Eraction} yields
\[
\begin{aligned}
(\mathsf E_r\mathsf E_s f)(x)
&= p_{x_r,x_{r+1}}(\mathsf E_s f)(x) + p_{x_{r+1},x_r}(\mathsf E_s f)(x^{\tau_r}) \\
&=p_{x_r,x_{r+1}}p_{x_s,x_{s+1}}\,f(x)
 +p_{x_r,x_{r+1}}p_{x_{s+1},x_s}\,f(x^{\tau_s})\\
&\qquad
 +p_{x_{r+1},x_r}p_{x_s,x_{s+1}}\,f(x^{\tau_r})
 +p_{x_{r+1},x_r}p_{x_{s+1},x_s}\,f(x^{\tau_s \tau_r}).
\end{aligned}
\]
Since $|r-s|\ge 2$, the transpositions $\tau_r$ and $\tau_s$ commute since they
 act on disjoint pairs of positions and hence  $\mathsf E_r\mathsf E_s=\mathsf E_s\mathsf E_r$.
\end{proof}

Next, we record an expression for the smallest positive eigenvalue of the sum of a consecutive pair of elementary transition matrices, which can be deduced from~\cite[Section ``Evidence in favor of the conjecture (b)'']{Fill}.

\begin{lemma}\label{lem:local-three-point}
Let $r \in \{1,\dots, n-2\}$.
Then the smallest positive eigenvalue of $\mathsf E_r+\mathsf E_{r+1}$ is $1-m_{\mathbf p}$.
\end{lemma}

Let $V$ be a subspace of $\mathbb{R}^{\mathfrak{S}_n}$.
A matrix $P$ is called an \emph{orthogonal projection} onto $V$ with respect to $\langle \cdot, \cdot \rangle$ if, for all $f \in \mathbb{R}^{\mathfrak{S}_n}$ and $g \in V$ we have $Pf \in V$, $Pg = g$, and $\langle f ,g \rangle = \langle Pf,g\rangle$.
It follows immediately that an orthogonal projection is idempotent.
Orthogonal projections must also be self-adjoint.
Indeed, let $f , g \in \mathbb{R}^{\mathfrak{S}_n}$.
One can orthogonally decompose $\mathbb{R}^{\mathfrak{S}_n} = V \oplus V^\perp$ so that $f = Pf + (f-Pf)$ and $g = Pg + (g-Pg)$.
Then $\langle Pf,g\rangle = \langle Pf,Pg\rangle + \langle Pf,g-Pg\rangle =  \langle Pf,Pg\rangle$ and similarly $\langle f,Pg\rangle = \langle Pf,Pg\rangle + \langle Pf-f,Pg\rangle =  \langle Pf,Pg\rangle$.

Define the subspace $V_r \subset \mathbb{R}^{\mathfrak{S}_n}$ by
\[
V_r := \{{f} \in \mathbb{R}^{\mathfrak{S}_n} :
{f}(x^{\tau_r}) = {f}(x)
\ \text{for all } x \in \mathfrak{S}_n\}.
\]

\begin{lemma}\label{lem:Qr-projection}
For each $1\le r\le n-1$, the $r$-th elementary transition matrix $\mathsf E_r$ is the orthogonal projection onto $V_r$ with respect to $\langle \cdot, \cdot \rangle$.
\end{lemma}

\begin{proof}
Let $f \in \mathbb R^{\Sn}$ and $x\in \mathfrak S_n$. 
Since $\mu(x^{\tau_r})/\mu(x) = p_{x_{r+1},x_r}/p_{x_r,x_{r+1}}$
and $p_{x_r,x_{r+1}}+p_{x_{r+1},x_r}=1$, we have
\begin{equation}
\label{eqn:Ef}
    \mathsf{E}_r f(x)
= p_{x_r,x_{r+1}} f(x) + p_{x_{r+1},x_r} f(x^{\tau_r})
= \frac{\mu(x)\,f(x)+\mu(x^{\tau_r})\,f(x^{\tau_r})}{\mu(x)+\mu(x^{\tau_r})}.
\end{equation}
Since this expression is symmetric under interchanging $x$ and $x^{\tau_r}$, we have
$\mathsf{E}_r f(x^{\tau_r}) = \mathsf{E}_r f(x)$.
Therefore, $\mathsf{E}_r f \in V_r$.
Conversely, if $f \in V_r$ then $f(x^{\tau_r}) = f(x)$ for all $x$.
Hence, \eqref{eqn:Ef} yields $\mathsf{E}_r f(x) = f(x)$.
Thus $\mathsf{E}_r$ maps $\mathbb{R}^{\mathfrak{S}_n}$ onto $V_r$ and acts as the identity on $V_r$.

Observe that
\[
(f(x)-\mathsf E_r f(x)){\mu}(x)
+(f(x^{\tau_r})-\mathsf E_r f(x^{\tau_r})){\mu}(x^{\tau_r})=0.
\]
For all $g \in V_r$, since $g(x) = g(x^{\tau_r})$, we have
\[
(f(x)-\mathsf E_r f(x))g(x){\mu}(x)
+(f(x^{\tau_r})-\mathsf E_r f(x^{\tau_r}))g(x^{\tau_r}){\mu}(x^{\tau_r})=0,
\]
and, by partitioning $\mathfrak S_n$ into pairs of the form $\{x,x^{\tau_r}\}$, we find that
\[
\langle f-\mathsf E_r f,\,g\rangle=\sum_{x\in \mathfrak S_n} (f(x)-\mathsf E_r f(x))g(x){\mu}(x) = 0.
\]
Therefore, $f-\mathsf E_r f$ is orthogonal to $V_r$, that is, $\mathsf E_r$ is the orthogonal projection of $\mathbb{R}^{\mathfrak{S}_n}$ onto $V_r$.
\end{proof}

\subsection{A lemma on pairs of orthogonal projections}

The following lemma, which is a crucial ingredient in the proof of Lemma~\ref{lem:A2-lower}, holds for any finite-dimensional real inner product space.
It provides, for any pair of orthogonal projections $P$ and $Q$, an upper bound on the angle between the vectors $Pf$ and $Qf$ in terms of the smallest non-zero eigenvalue of $P+Q$.

The \emph{operator norm} $\| R \|$ induced by $\langle \cdot, \cdot \rangle$ is defined by
\[
\| R \| := \max_{\| f \|=1} \| Rf \| =  \max_{\| f \|=1} \sqrt{\langle Rf, Rf \rangle} =  \max_{\| f \|=1} \sqrt{\langle R^*Rf, f \rangle},
\]
where $R^*$ denotes the adjoint of $R$ with respect to $\langle \cdot, \cdot \rangle$.
It follows that $\| R \|$ is equal to the maximum singular value $\sigma$ of $R$, that is, the square root of the maximum eigenvalue of $R^*R$.

\begin{lemma}\label{lem:two-projections}
Let $P$ and $Q$ be orthogonal projections on a finite-dimensional real inner product space $V$.
Suppose that every non-zero eigenvalue of $P+Q$ is at least $1-p$, where $0\le p<1$.
Then, for every $f\in V$,
\[
\langle Pf,Qf\rangle\ge -p\,\|Pf\|\,\|Qf\|.
\]
In the case of equality, we have $f \!\in\! \bigl(\operatorname{Im}(P) \cap \operatorname{Im}(Q)\bigr)^\perp$, and when $Pf\neq 0\neq Qf$, one has $PQPf = p^2Pf$.
\end{lemma}
\begin{proof}
We may assume that $P$ and $Q$ are both non-zero orthogonal projections, since otherwise the conclusion is immediate.
First, we orthogonally decompose the space as $V = W \oplus W^\perp$, where $W = \operatorname{Im}(P) \cap \operatorname{Im}(Q)$.
For any $f \in V$, we can write $f = f_1 + f_2$ where $f_1 \in W$ and $f_2 \in W^\perp$.
Since $W$ is invariant under $P$ and $Q$, its orthogonal complement $W^\perp$ is also invariant under both self-adjoint projections. 
Thus, $Pf_2, Qf_2 \in W^\perp$.
We have $Pf = f_1 + Pf_2$ and $Qf = f_1 + Qf_2$. 
Observe that $\langle f_1, Pf_2 \rangle = \langle Pf_1, f_2 \rangle = \langle f_1, f_2 \rangle = 0$ and similarly $\langle f_1, Qf_2 \rangle = 0$.
Thus,
\[
\langle Pf, Qf \rangle = \|f_1\|^2 + \langle Pf_2, Qf_2 \rangle.
\]
Similarly, $\|Pf\|^2 = \|f_1\|^2 + \|Pf_2\|^2$ and $\|Qf\|^2 = \|f_1\|^2 + \|Qf_2\|^2$.
Suppose we can establish our bound on $W^\perp$, namely $\langle Pf_2, Qf_2 \rangle \ge -p \|Pf_2\| \|Qf_2\|$. 
Since $p \ge 0$, we would have
\begin{equation}
    \label{eqn:f1}
    \langle Pf, Qf \rangle \ge \|f_1\|^2 - p \|Pf_2\| \|Qf_2\| \ge
-p \|Pf\| \|Qf\|.
\end{equation}
It therefore suffices to show that $\langle Pf, Qf \rangle \ge -p \|Pf\| \|Qf\|$ for $f \in W^\perp$.

Denote by $R$ the operator $PQ$ restricted to $W^\perp$, that is, $R = PQ|_{W^\perp}$.
For $f \in W^\perp$, using the fact that $P$ and $Q$ are idempotent and self-adjoint with respect to $\langle \cdot,\cdot\rangle$ together with Cauchy--Schwarz yields
\begin{equation}
\label{eqn:cs}
    |\langle Pf, Qf \rangle| = |\langle P^2f, Q^2 f \rangle| =  |\langle Pf, (PQ)Q f \rangle| \le \|R\| \|Pf\| \|Qf\|.
\end{equation}
It therefore suffices to show that the restricted operator norm satisfies $\|R\| \le p$.

Let $f \in W^\perp$ be a unit vector such that $\|R\| = \|Rf\| = \sigma >0$.
Now set $g = \sigma^{-1}PQf$.
Then $\| g \| = 1$ and $Pg = g$, since $g \in \text{Im}(P)$.
Since both $P$ and $Q$ are self-adjoint with respect to $\langle \cdot,\cdot\rangle$, the adjoint of $R$ is $R^* = QP|_{W^\perp}$.
Indeed, for all $u,v \in W^\perp$, we have $\langle Ru,v\rangle = \langle PQu,v\rangle = \langle u,QPv\rangle$.
Next, we have $Qg = QPg = \sigma^{-1}R^*Rf = \sigma f$.
Thus, $f \in \operatorname{Im}Q$ and, since $Q$ is idempotent, we have $Qf = f$.
Multiplying both sides by $P$ yields $Pf = PQf = \sigma g$.
Therefore,
\[
(P+Q)(g- f) = (g-\sigma g) + (\sigma f- f) = (1-\sigma)(g - f).
\]
 
If $g =  f$, then $P g =  g$ and $Q g =  g$, meaning $g \in \text{Im}(P) \cap \text{Im}(Q) = W$, which leads to a contradiction since $g \in W^\perp$ is a unit vector. 
Thus, $ g \neq  f$.

Since $g \neq  f$, the vector $g - f$ is a non-zero eigenvector of $P+Q$ with eigenvalue $1-\sigma$.
We claim that $1-\sigma$ is non-zero. 
Indeed, observe that
\[
\sigma = \langle g, \sigma g \rangle = \langle  g, P f \rangle = \langle P g,  f \rangle = \langle g,  f \rangle.
\]
By Cauchy--Schwarz, since $g$ and $f$ are distinct unit vectors, $\langle g,  f \rangle < 1$. 
Thus, $\sigma < 1$, and $1-\sigma$ is strictly positive.

By assumption, every non-zero eigenvalue of $P+Q$ is at least $1-p$. 
Since $1-\sigma$ is a non-zero eigenvalue of $P+Q$, we have $1-\sigma \ge 1-p$, which implies $\sigma \le p$.
Thus, $\|R\| \le p$, as required.

Now, consider the case of equality, that is, $\langle Pf,Qf\rangle = -p\,\|Pf\|\,\|Qf\|$.
From \eqref{eqn:f1}, we have
\[
\|f_1\|^2 =
p(\|Pf_2\| \|Qf_2\|- \|Pf\| \|Qf\|) \le 0.
\]
Hence, $f_1 = 0$ and $f \in W^\perp$.
Furthermore, we must also have $Qf , PQf \in W^\perp$.
Since $\langle Pf, PQf \rangle = \langle Pf, Qf \rangle$, equality in the Cauchy--Schwarz inequality \eqref{eqn:cs} yields
\begin{equation}
    \label{eqn:PQf}
    PQf = -p\frac{\|Qf\|}{\| Pf \|}Pf.
\end{equation}
Exchanging the roles of $P$ and $Q$, one also finds
\[
QPf = -p\frac{\|Pf\|}{\| Qf \|}Qf,
\]
which, when combined with \eqref{eqn:PQf}, yields $PQPf = p^2Pf$, as required.
\end{proof}

\section{Spectral gap and multiplicities}

\subsection{Lower bound on spectral gap}

By definition, we can express the transition matrix $K$ as the average of the elementary transition matrices, that is, 
$K = \frac{1}{n-1}\sum_{r=1}^{n-1} \mathsf E_r$.
We now prove a lower bound on $\langle f,K^2f\rangle$ and characterise equality in terms of elementary transition matrices.
This lemma plays a central role in the proof of our main results.

\begin{lemma}\label{lem:A2-lower}
Let $f\in \mathbb{R}^{\mathfrak{S}_n}$.
Then, for $n\ge 3$,
\begin{equation*}
\langle f,K^2f\rangle\ge \frac{1-2m_{\mathbf p}\cos(\tfrac{{\pi}}{n})}{n-1}\langle f, Kf\rangle.
\end{equation*}
In the case of equality, there exists a non-negative constant $c$ such that
\begin{align}
    \mathsf E_r \mathsf E_s f&= 0 \qquad\qquad\quad \ \ \;\text{ for $|r-s| \ge 2$;} \label{eqn:E1} \\
    \mathsf E_r \mathsf E_{r+1} \mathsf E_r f &= m_{\mathbf p}^2 \mathsf E_r f  \qquad\quad\ \text{ for each $r \in \{1,\dots, n-2\}$;} \label{eqn:E2}\\
    \| \mathsf E_r f\| &= c \sin (r \pi/n) \quad \ \  \text{ for each $r \in \{1,\dots, n-1\}$.} \label{eqn:E3}
\end{align}
\end{lemma}

\begin{proof}
Since $K=\frac{1}{n-1}\sum_{r=1}^{n-1}\mathsf E_r$ and each $\mathsf E_r$ is self-adjoint, expanding the inner product yields
\begin{align*}
\langle f,K^2 f\rangle
&=\frac{1}{(n-1)^2}\left (\sum_{r=1}^{n-1} \|\mathsf E_r f\|^2
+2\sum_{r=1}^{n-2}\langle \mathsf E_r f,\mathsf E_{r+1}f\rangle
+2\sum_{\substack{1\le r<s\le n-1\\ |r-s|\ge 2}}
\langle \mathsf E_r f,\mathsf E_s f\rangle \right ).
\end{align*}

By Lemma~\ref{lem:far-commute}, if $|r-s|\ge 2$, then since $\mathsf E_r$ and $\mathsf E_s$ are commuting orthogonal projections with respect to $\langle \cdot,\cdot\rangle$, it follows that the matrix $\mathsf E_r\mathsf E_s$ is idempotent and self-adjoint with respect to $\langle \cdot,\cdot\rangle$.
Hence,
\[
\langle \mathsf E_r f,\mathsf E_s f\rangle
=\langle f,\mathsf E_r\mathsf E_s f\rangle
=\|\mathsf E_r\mathsf E_s f\|^2
\ge 0.
\]

By Lemma~\ref{lem:local-three-point}, every non-zero eigenvalue of
$\mathsf E_r+\mathsf E_{r+1}$ is at least $1-m_{\mathbf p}$. 
Applying
Lemma~\ref{lem:two-projections} with $P=\mathsf E_r$ and
$Q=\mathsf E_{r+1}$ and $p=m_{\mathbf p}$, we obtain, for $1\le r\le n-2$,
\[
\langle \mathsf E_r f,\mathsf E_{r+1}f\rangle
\ge -m_{\mathbf p}\,\|\mathsf E_r f\| \|\mathsf E_{r+1} f\|.
\]
Therefore,
\begin{equation}
    \label{eqn:rhs}
    \langle f,K^2 f\rangle
\ge
\frac{1}{(n-1)^2}\left (\sum_{r=1}^{n-1} \|\mathsf E_r f\|^2
-2m_{\mathbf p}\sum_{r=1}^{n-2} \|\mathsf E_r f\|\|\mathsf E_{r+1} f\| \right ).
\end{equation}

Set $
v_r:=\|\mathsf E_r f\|$ for each $r\in \{1,\dots,n-1\}$, and write $\mathbf v=(v_1,\dots,v_{n-1})^\top\in\mathbb R^{n-1}$. 
Then the right-hand side of \eqref{eqn:rhs} equals $\mathbf v^\top T\mathbf v$,
where
\[
T:=
\frac{1}{(n-1)^2}\begin{bmatrix}
1&-m_{\mathbf p}&&&\\
-m_{\mathbf p}&1&-m_{\mathbf p}&&\\
&\ddots&\ddots&\ddots&\\
&&-m_{\mathbf p}&1&-m_{\mathbf p}\\
&&&-m_{\mathbf p}&1
\end{bmatrix}.
\]
The smallest eigenvalue of $T$ is $(1-2m_{\mathbf p}\cos(\tfrac \pi n))/(n-1)^2$ (see~\cite[Section 1.4.4]{spectra}). 
Hence
\[
\mathbf v^\top T\mathbf v
\ge
\frac{1-2m_{\mathbf p}\cos(\tfrac {\pi} n)}{(n-1)^2}\sum_{r=1}^{n-1} \|\mathsf E_r f\|^2.
\]

Finally, since each $\mathsf E_r$ is an orthogonal projection,
\[
\sum_{r=1}^{n-1} \|\mathsf E_r f\|^2
=\sum_{r=1}^{n-1}\langle f,\mathsf E_r f\rangle
=(n-1)\langle f,K f\rangle.
\]
This proves the inequality.

In the case of equality, using \eqref{eqn:rhs}, we find that
\begin{align}
\label{eqn:eq1}
    \sum_{\substack{1\le r<s\le n-1\\ |r-s|\ge 2}}
\|\mathsf E_r\mathsf E_s f\|^2 &= 0; \\
\label{eqn:eq2}
\sum_{r=1}^{n-2}
\Bigl(
m_{\mathbf p}\|\mathsf E_r f\|\,\|\mathsf E_{r+1}f\|
+
\langle \mathsf E_r f,\mathsf E_{r+1}f\rangle
\Bigr) &= 0; \\
\label{eqn:eq3}
\left(
\mathbf v^{\!\top}T\mathbf v
-
\frac{1-2m_{\mathbf p}\cos(\pi/n)}{(n-1)^2}
\sum_{r=1}^{n-1} v_r^2
\right) &= 0.
\end{align}
The first equation, \eqref{eqn:eq1}, implies~\eqref{eqn:E1}.
The last equation, \eqref{eqn:eq3}, implies that either $\mathbf v=0$, in which case \eqref{eqn:E2} and \eqref{eqn:E3} are trivial, 
or else $\mathbf v$ is an eigenvector of $T$ corresponding to its smallest eigenvalue.
Since this eigenspace is spanned by $(\sin(r\pi/n))_{r=1}^{n-1}$, we have \eqref{eqn:E3}. 
Finally, by Lemma~\ref{lem:two-projections}, the second equation \eqref{eqn:eq2} implies~\eqref{eqn:E2}. 
\end{proof}

Taking $f$ to be an eigenvector for a non-zero eigenvalue of $K$ in Lemma~\ref{lem:A2-lower} yields the following corollary, which implies Theorem~\ref{thm:main} using the similarity of $K$ and $I-K$~\cite[Theorem~2]{Fill}.

\begin{corollary}\label{cor:main}
Each positive eigenvalue of $K$ is at least $(1-2m_{\mathbf p}\cos(\tfrac{\pi}{n}))/(n-1)$.
\end{corollary}

\subsection{Probability vectors with minimum spectral gap}

In this section, we consider the case of equality in Corollary~\ref{cor:main}.

Given a probability vector $\mathbf p$, define the set
\[
H_{\mathbf p} = \{ i \in [n] \;:\; p_{i,j} = 1/2 \text{ for all } j \ne i\}.
\]
As we shall see below, regular probability vectors $\mathbf p$ that achieve the minimum spectral gap share the property that the set $H_{\mathbf p}$ is non-empty.

\begin{lemma}\label{lem:interval}
Let $\mathbf p$ be a regular probability vector.
    Then the set $H_{\mathbf p}$ is an interval. 
    Furthermore, if $H_{\mathbf p}=[1,b]$ for some $b<n$, or $H_{\mathbf p}=[a,n]$ for some $a>1$, then $\mathbf p$ is uniform. 
\end{lemma}

\begin{proof}
    Let $i<j<k$ and suppose that $i,k\in H_{\mathbf p}$. Fix $\ell>j$. Since $i\in H_{\mathbf p}$, we have $p_{i,\ell}=1/2$.
By repeated use of \eqref{eq:regular2}, we have $p_{j,\ell}\le p_{i,\ell}=1/2$.
On the other hand, repeated use of \eqref{eq:regular3} and then \eqref{eq:regular1} gives $p_{j,\ell}\ge p_{j,j+1}\ge  1/2$.
Hence $p_{j,\ell}=1/2$ for every $\ell>j$.

Now fix $\ell<j$.
Since $k\in H_{\mathbf p}$, we have $p_{\ell,k}=1/2$.
By repeated use of \eqref{eq:regular3}, $p_{\ell,j}\le p_{\ell,k}=1/2$, whereas repeated use of \eqref{eq:regular2} and then \eqref{eq:regular1} gives $p_{\ell,j}\ge p_{j-1,j}\ge 1/2$.
Thus, $p_{\ell,j}=1/2$ for every $\ell<j$, and hence $j\in H_{\mathbf p}$.
Therefore, $H_{\mathbf p}$ is an interval.

For the second statement, we prove only the first case, the second being analogous. Suppose that $H_{\mathbf p}=[1,b]$ with $b<n$. Let $b<u<v\le n$.
Since $b\in H_{\mathbf p}$, we have $p_{b,v}=1/2$.
Repeated use of \eqref{eq:regular2} gives $p_{u,v}\le p_{b,v}=1/2$, while repeated use of \eqref{eq:regular3} and then \eqref{eq:regular1} gives $p_{u,v}\ge p_{u,u+1}\ge 1/2$.
Hence $p_{u,v}=1/2$ for all $b<u<v\le n$ and it follows that $\mathbf p$ is uniform.
\end{proof}

Define $z_{i,j} = (i,j,\xi_3,\dots,\xi_n) \in \mathfrak S_n$ to be the permutation such that $i$ is in the first position, $j$ is in the second position and $\xi_3, \dots, \xi_n$ are in increasing order.
Define the set $S_f$ by 
\[
S_f := \{(i,j) \; : \; 1 \le i < j \le n \text{ and } \mathsf E_1 f(z_{i,j}) \ne 0 \}.
\]

\begin{lemma}\label{lem:tech}
    Let $n\ge 3$.
    Suppose that the smallest positive eigenvalue of $K$ is $(1-2m_{\mathbf p}\cos(\tfrac{{\pi}}{n}))/(n-1)$ with corresponding eigenvector $f$.
    Then $S_f$ is non-empty and
    \[
    \sqrt{p_{i,j}p_{j,k}p_{k,i} + p_{k,j}p_{j,i}p_{i,k}} = m_{\mathbf p}
    \]
    for all $i<j<k$ such that $\{(i,j),(i,k),(j,k)\} \cap S_f \ne \emptyset$.
    Furthermore, if $\mathbf p$ is regular and $m_{\mathbf p}=1/2$, then 
    the lexicographically minimal element of $S_f$ has the form $(1,b)$ and, if $b = n$ then $[2,n-1] \subseteq H_{\mathbf p}$, otherwise $b \in H_{\mathbf p}$.
\end{lemma}
\begin{proof}
    Let $f \ne 0$ be an eigenvector of $K$ with eigenvalue $(1-2m_{\mathbf p}\cos(\tfrac{{\pi}}{n}))/(n-1)$.
    Using \eqref{eqn:E3} in Lemma~\ref{lem:A2-lower}, we can write $\|\mathsf E_r f\| = c \sin(r \pi/n)$ for each $r \in \{1,\dots,n-1\}$.
    Since $\langle f,Kf \rangle > 0$, we must have $c > 0$.
    Furthermore, we apply \eqref{eqn:E1} and \eqref{eqn:E2} when $r = 1$ to deduce
    \begin{align}
        \mathsf E_1 \mathsf E_s f &= 0 \qquad \quad \text{ for all $s \ge 3$}; \label{eqn:E11} \\
         \mathsf E_1 \mathsf E_2 \mathsf E_1 f &= m_{\mathbf p}^2 \mathsf E_1 f. \label{eqn:E21}
    \end{align}
    We now show that $S_f$ is nonempty.
    Since $c >0$, there exists $x \in \mathfrak{S}_n$ such that $\mathsf E_1 f(x) \ne 0$.
    We may assume that the first two labels in $x$ are increasing since $\mathsf E_1 f(x^{\tau_1}) = \mathsf E_1 f(x) \ne 0$.
    Furthermore, using \eqref{eqn:E11}, we can deduce
    \begin{equation}
        \label{eqn:prop}
        \mathsf E_1 f(x^{\tau_s}) = -\frac{p_{x_s,x_{s+1}}}{p_{x_{s+1},x_{s}}}\mathsf E_1 f(x).
    \end{equation}
    Repeatedly applying \eqref{eqn:prop} as necessary to arrange the elements of $x$ from positions $3$ to $n$ in increasing order establishes our claim.

    Fix $i < j < k$ and write $m_{i,j,k} = \sqrt{p_{i,j}p_{j,k}p_{k,i} + p_{k,j}p_{j,i}p_{i,k}}$.
    Define $y_{i,j,k} = (i,j,k,\xi_4,\dots,\xi_n) \in \mathfrak S_n$ to be the permutation such that the first three positions are  $(i,j,k)$ and $\xi_4, \dots, \xi_n$ are in increasing order.
    Applying \eqref{eqn:E21}, we obtain the 
    %
%
vector equation
\begin{equation}
    \label{eqn:system}
     \begin{bmatrix}
    \mathsf E_1 f(y_{i,j,k}) & \mathsf E_1 f(y_{i,k,j}) & \mathsf E_1 f(y_{j,k,i})
\end{bmatrix}M = \mathbf 0,
\end{equation}
where
\[
    M = (m_{i,j,k}^2 - m_{\mathbf p}^2)I + \begin{bmatrix}
    p_{j,k}p_{i,k} & p_{j,k}p_{i,k} &  p_{j,k}p_{i,k}\\
    p_{i,j}p_{k,j}& p_{i,j}p_{k,j} & p_{i,j}p_{k,j} \\
    p_{j,i}p_{k,i} & p_{j,i}p_{k,i} & p_{j,i}p_{k,i}
\end{bmatrix}.
\]
Using \eqref{eqn:prop}, if any of $(i,j)$, $(i,k)$, or $(j,k)$ belong to $S_f$ then the vector $\left [ \begin{smallmatrix}
    \mathsf E_1 f(y_{i,j,k}) & \mathsf E_1 f(y_{i,k,j}) & \mathsf E_1 f(y_{j,k,i})
\end{smallmatrix} \right ]$ is non-zero, and thus a nullvector for $M$.
In such cases, we have $\det M = (1-m_{\mathbf p}^2)(m_{i,j,k}^2-m_{\mathbf p}^2)^2 = 0$.
Since $m_{\mathbf p} < 1$, we must have $m_{i,j,k} = m_{\mathbf p}$, as required.

Now suppose that $\mathbf p$ is regular and that $m_{\mathbf p}=1/2$.
It is straightforward to show that $m_{i,j,k} = 1/2$ implies $p_{i,j} = p_{j,k} = 1/2$.
Moreover, when $m_{i,j,k} = 1/2$, \eqref{eqn:system} reduces to 
\begin{equation}
\label{eqn:reduced}
    \frac{p_{i,k}}{2}\mathsf E_1 f(y_{i,j,k}) + \frac 1 4 \mathsf E_1 f(y_{i,k,j}) + \frac{p_{k,i}}{2} \mathsf E_1 f(y_{j,k,i}) = 0.
\end{equation}
Since all the coefficients of \eqref{eqn:reduced} are positive, if $(i,j) \in S_f$ then at least one of $(i,k)$ or $(j,k)$ must also be in $S_f$.
We now claim that the lexicographically minimal element of $S_f$ has the form $(1,b)$.
Suppose for a contradiction that $(a,b) \in S_f$ is the lexicographically minimal element with $a > 1$.
Consider the triple $i < j < k$ with $i = a-1$, $j = a$, and $k = b$.
Since $(j,k) \in S_f$, we have $m_{i,j,k} = 1/2$ and the vector $\left[\begin{smallmatrix} \mathsf E_1 f(y_{i,j,k}) & \mathsf E_1 f(y_{i,k,j}) & \mathsf E_1 f(y_{j,k,i}) \end{smallmatrix}\right]$ is non-zero.
Using \eqref{eqn:reduced}, at least one of $\mathsf E_1 f(y_{i,j,k})$ and $\mathsf E_1 f(y_{i,k,j})$ is non-zero.
By repeated use of \eqref{eqn:prop}, it follows that $(i,j)=(a-1,a)\in S_f$ or $(i,k)=(a-1,b)\in S_f$, which contradicts $(a,b)$ being lexicographically minimal.
Hence $a=1$.

Let $(1,b)$ be the lexicographically minimal element of $S_f$.
Firstly, suppose $b=n$.
Then, for each $a\in\{2,\dots,n-1\}$, since $(1,n)\in S_f$ we have $m_{1,a,n}=1/2$, hence $p_{1,a}=p_{a,n}=1/2$.
Now fix $a\in\{2,\dots,n-1\}$. 
If $r>a$, then repeated use of \eqref{eq:regular2} gives $p_{r-1,r}\le p_{a,r}$, while repeated use of \eqref{eq:regular3} gives $p_{a,r}\le p_{a,n}=1/2$.
Since \eqref{eq:regular1} gives $1/2\le p_{r-1,r}$, it follows that $1/2\le p_{r-1,r}\le p_{a,r}\le p_{a,n}=1/2$, so $p_{a,r}=1/2$. 
Similarly, if $r<a$, then repeated use of \eqref{eq:regular3} gives $p_{r,r+1}\le p_{r,a}$, while repeated use of \eqref{eq:regular2} gives $p_{r,a}\le p_{1,a}=1/2$.
Since \eqref{eq:regular1} gives $1/2\le p_{r,r+1}$, we obtain $1/2 \le p_{r,r+1}\le p_{r,a}\le p_{1,a}=1/2$, so $p_{r,a}=1/2$, and hence also $p_{a,r}=1/2$.
Therefore, $p_{a,r}=1/2$ for all $a\in\{2,\dots,n-1\}$ and all $r\ne a$, that is, $[2,n-1] \subseteq H_{\mathbf p}$. 

Finally, suppose $b < n$.
Consider $a \in \{2,\dots,b-1\}$.
Since $(1,b) \in S_f$, we have $m_{1,a,b} = 1/2$, which implies $p_{a,b} = p_{b,a} =1/2$.
Similarly, for $a \in \{b+1,\dots,n\}$, we have $m_{1,b,a} = 1/2$, which implies $p_{1,b} = p_{b,1} = 1/2$ and $p_{b,a} = 1/2$, thus completing the proof.
\end{proof}

Next, we prove a partial converse of Lemma~\ref{lem:tech} for the case when the probability vector $\mathbf p$ is regular.
Denote by $\mathfrak f_i \in \mathbb R^{\mathfrak S_n}$ the vector defined by
\[
\mathfrak f_i(x) = \cos \left ( \frac{(x^{-1}(i)-1/2)\pi}{n}   \right ).
\]
The vectors $\mathfrak f_i$ were originally identified by Wilson, who used these vectors to lower bound the multiplicity of the second-largest eigenvalue of $K$ in the uniform case \cite[Theorem 14]{W}.
For notional convenience, we define the function $\phi$ by $\phi(r):=\cos\!\left(\frac{(2r-1)\pi}{2n}\right)$ and record resulting standard trigonometric identities:
\begin{align}
\label{eqn:trig1}
        \phi(r-1)+\phi(r+1)&=2\cos\!\left(\frac{\pi}{n}\right)\phi(r), \quad \text{ for all $r \in \mathbb Z$}; \\
\label{eqn:trig2}
\begin{split}
    \phi(0) &= \phi(1);\\
   \phi(n+1)
&=
\phi(n).
\end{split}
\end{align}

\begin{lemma}\label{lem:assume_1/2}
Let $\mathbf p$ be a probability vector.  
    Then $\mathfrak f_i$ is an eigenvector of $K$ with eigenvalue $\left (1-\frac{1-\cos(\pi/n)}{n-1}\right )$ for each $i \in H_{\mathbf p}$.
\end{lemma}

\begin{proof}


Fix $x \in \mathfrak S_n$ and let $s = x^{-1}(i)$.
Applying \eqref{eqn:Eraction} yields
\[
\mathsf E_r \mathfrak f_i(x) = \begin{cases}
    \phi(s)  & \text{ if $r \not \in \{s-1,s\}$} \\
    \phi(s)/2 + \phi(s-1)/2 & \text{ if $r = s-1$} \\
    \phi(s)/2 + \phi(s+1)/2 & \text{ if $r = s$}.
\end{cases}
\]
Since $K$ is the average of the elementary transition matrices, we have
\[
K\mathfrak f_i(x) = \begin{cases} \frac{1}{n-1}\left ( (n-2)\phi(s) + \phi(s-1)/2 +   \phi(s+1)/2 \right ) & \text{ if $s \not \in \{1,n\}$}; \\
 \frac{1}{n-1}\left ( (n-2)\phi(1) + \phi(1)/2 +   \phi(2)/2 \right ) & \text{ if $s = 1$}; \\
 \frac{1}{n-1}\left ( (n-2)\phi(n) + \phi(n)/2 +   \phi(n-1)/2 \right ) & \text{ if $s = n$}.
\end{cases}
\]
Hence, using \eqref{eqn:trig1} and \eqref{eqn:trig2}, for all $s \in \{1,\dots,n\}$, we have
\begin{align*}
    K\mathfrak f_i(x) &= \frac{1}{n-1}\left ( (n-2)\phi(s) + \phi(s-1)/2 +   \phi(s+1)/2 \right ) \\
    &= \frac{1}{n-1}\left ( (n-2)\phi(s) + \phi(s)  \cos \left ( \frac{\pi}{n} \right ) \right ) \\
    &= \left (1 - \frac{1-\cos \left ( \frac{\pi}{n} \right )}{n-1} \right )\mathfrak f_i(x),
\end{align*}
as required.
\end{proof}

Now we can prove Theorem~\ref{thm:Fillc1}.

\begin{proof}[Proof of {Theorem~\ref{thm:Fillc1}}]
    Lemma~\ref{lem:assume_1/2} together with Corollary~\ref{cor:reg} imply one direction of Theorem~\ref{thm:Fillc1}.
For the other direction, Corollary~\ref{cor:reg} implies that $m_{\mathbf p} = 1/2$ and the theorem follows from Lemma~\ref{lem:tech}.
\end{proof}

\subsection{Multiplicity of the second-largest eigenvalue}

In this section, we consider the multiplicity of the second-largest eigenvalue of $K$ in the case when the probability vector $\mathbf p$ is regular, and the spectral gap of $K$ is minimal, that is, $\lambda_K=(1-\cos(\pi/n))/(n-1)$.

Define $\mathfrak g\in\mathbb R^{\mathfrak S_n}$ by
\[
\mathfrak g(x):=
\begin{cases}
\phi(x^{-1}(1))-\phi(x^{-1}(n)),
& \text{ if } x^{-1}(1)<x^{-1}(n),\\
\frac{p_{1,n}}{p_{n,1}}\left ( \phi(x^{-1}(1))-\phi(x^{-1}(n))\right ),
& \text{ if } x^{-1}(1)>x^{-1}(n).
\end{cases}
\]

\begin{lemma}
\label{lem:g}
    Let $n \ge 3$ and let $\mathbf p$ be a regular probability vector. 
    Suppose $|H_{\mathbf p}| = n-2$.
    Then $\mathfrak g$ is an eigenvector of $K$ with eigenvalue $\left (1-\frac{1-\cos(\pi/n)}{n-1}\right )$.
\end{lemma}
\begin{proof}
By Lemma~\ref{lem:interval}, $p_{i,j} = 1/2$ unless $\{i,j\} = \{1,n\}$.
Fix $x\in \Sn$, and write $s:=x^{-1}(1)$ and $t:=x^{-1}(n)$.
Then $\mathfrak g(x)=\varepsilon(x)\bigl(\phi(s)-\phi(t)\bigr)$ where
\[
\varepsilon(x):=
\begin{cases}
1,& \text{ if }s<t,\\
p_{1,n}/p_{n,1},& \text{ if } s>t.
\end{cases}
\]

The product $\mathsf E_r \mathfrak g(x)$ depends on $r$, $s$, and $t$.
Firstly, if $\{r,r+1\} \cap \{s,t\} = \emptyset$ then $\mathsf E_r \mathfrak g(x) = \mathfrak g(x)$.
Indeed, neither of the labels $1$ and $n$ is moved and hence $\varepsilon (x) = \varepsilon (x^{\tau_r})$ and $\mathfrak g(x) = \mathfrak g(x^{\tau_r})$.

Now suppose $s = r$ and $t \ne r+1$.
We claim 
\[
\mathsf E_r\mathfrak g(x)
=
\varepsilon(x)\left(\frac{\phi(s)+\phi(s+1)}{2}-\phi(t)\right).
\]
Indeed, the label $1$ is at position $r=s$, and some label $j\in H_{\mathbf p}$ is at position $r+1$. 
After applying $\tau_r$, the label $1$ moves to $s+1\neq t$, so the relative order of $1$ and $n$ is unchanged. 
Hence, $\varepsilon(x^{\tau_r})=\varepsilon(x)$.

Similarly, we find
\[
\mathsf E_r\mathfrak g(x)
= \begin{cases}
    \varepsilon(x)\left(\frac{\phi(s-1)+\phi(s)}{2}-\phi(t)\right) & \text{ if $s = r+1$ and $t \ne r$}; \\
    \varepsilon(x)\left(\phi(s)-\frac{\phi(t)+\phi(t+1)}{2}\right) & \text{ if $t = r$ and $s \ne r+1$}; \\
    \varepsilon(x)\left(\phi(s)-\frac{\phi(t-1)+\phi(t)}{2}\right) & \text{ if $t = r+1$ and $s \ne r$}.
\end{cases}
\]
Now, consider the case $\{r,r+1\}=\{s,t\}$.
Suppose first that $s<t$, so $s=r$ and $t=r+1$. 
Then $x_r=1$ and $x_{r+1}=n$, and
\[
\mathfrak g(x^{\tau_r})
=
\frac{p_{1,n}}{p_{n,1}}\bigl(\phi(t)-\phi(s)\bigr)
=
-\frac{p_{1,n}}{p_{n,1}}\,\mathfrak g(x).
\]
Hence,
\[
\mathsf E_r\mathfrak g(x)
=
p_{1,n}\mathfrak g(x)+p_{n,1}\mathfrak g(x^{\tau_r})
=
p_{1,n}\mathfrak g(x)-p_{1,n}\mathfrak g(x)
=
0.
\]
The case $s>t$ similarly concludes with $\mathsf E_r\mathfrak g(x) = 0$.

The cases when $|s-t|=1$ and $|s-t| \ge 2$ behave differently.  
First, we consider the case when $|s-t|=1$.
We will assume that $t=s+1$ since the case when $s = t+1$ is analogous.
Then, as shown above, $\mathsf E_s\mathfrak g(x)=0$, 
\[
\mathsf E_t\mathfrak g(x)
=
\varepsilon(x)\left(\phi(s)-\frac{\phi(t)+\phi(t+1)}{2}\right)
\]
and, for all $r\notin\{s-1,s,t\}$, we have
\[
\mathsf E_r\mathfrak g(x)=\mathfrak g(x).
\]
If $s=1$ and $t=2$, then $\mathsf E_1\mathfrak g(x)=0$, $\mathsf E_r\mathfrak g(x)=\mathfrak g(x)$ for all $r\ge 3$, and
\[
\mathsf E_2\mathfrak g(x)
=
\varepsilon(x)\left(\phi(1)-\frac{\phi(2)+\phi(3)}{2}\right).
\]
One can then check, using \eqref{eqn:trig1} and \eqref{eqn:trig2}, that
\[
K\mathfrak g(x)
=
\frac{n-2+\cos(\pi/n)}{n-1}\,\mathfrak g(x)
=
\left(1-\frac{1-\cos(\pi/n)}{n-1}\right)\mathfrak g(x),
\]
as required.
Otherwise, for $s\ge 2$, we have
\[
\mathsf E_{s-1}\mathfrak g(x)
=
\varepsilon(x)\left(\frac{\phi(s-1)+\phi(s)}{2}-\phi(t)\right)
\]
and
\[
\mathsf E_t\mathfrak g(x)
=
\varepsilon(x)\left(\phi(s)-\frac{\phi(t)+\phi(t+1)}{2}\right).
\]
Hence, if $s\ge 2$, then
\[
\sum_{r=1}^{n-1}\mathsf E_r\mathfrak g(x)
= (n-2)\mathfrak g(x) + 
\varepsilon(x)\left(
\frac{\phi(s-1)+\phi(s+1)}{2}
-\frac{\phi(t-1)+\phi(t+1)}{2}
\right),
\]
where we used $\phi(t-1)=\phi(s)$ and $\phi(s+1)=\phi(t)$.

Finally, we consider the case when $|s-t|\ge 2$.
Summing over all $r \in \{1,\dots,n-1\}$ yields
the same equation as above, that is,
\[
    \sum_{r=1}^{n-1}\mathsf E_r\mathfrak g(x)
= (n-2)\mathfrak g(x) + 
\varepsilon(x)\left (
\frac{\phi(s-1)+\phi(s+1)}{2}
-\frac{\phi(t-1)+\phi(t+1)}{2}
\right ),
\]
where we use \eqref{eqn:trig2} in the cases when $\{s,t\} \cap \{1,n\} \ne \emptyset$.
Using \eqref{eqn:trig1}, the above sum becomes
\[
\sum_{r=1}^{n-1}\mathsf E_r\mathfrak g(x)
= (n-2)\mathfrak g(x) +
\varepsilon(x)\left[
\cos\!\left(\frac{\pi}{n}\right)\phi(s)
-
\cos\!\left(\frac{\pi}{n}\right)\phi(t)
\right].
\]
Therefore,
\[
K\mathfrak g(x)
=
\frac{n-2+\cos(\pi/n)}{n-1}\,\mathfrak g(x)
=
\left(1-\frac{1-\cos(\pi/n)}{n-1}\right)\mathfrak g(x),
\]
as required.
\end{proof}

Now we are ready to prove the main result of this section, which provides an expression for the multiplicity of the second-largest eigenvalue of $K$ for extremal regular probability vectors.
The following theorem implies Theorem~\ref{thm:Fillc2}.
Indeed, the two formulations are equivalent by the similarity of $K$ and $I-K$, which is also used in the proof.

\begin{theorem}\label{thm:mult}
Let $n \ge 3$ and let $\mathbf p$ be a regular probability vector.
    Suppose $\lambda_K=(1-\cos(\pi/n))/(n-1)$.
    Then
        \[
\dim \ker \left (K-\lambda_K I \right ) =
\begin{cases}
|H_{\mathbf p}|, & |H_{\mathbf p}|\notin\{n,n-2\},\\
n-1, & |H_{\mathbf p}|\in\{n,n-2\}.
\end{cases}
    \]
\end{theorem}
\begin{proof}
We begin by proving lower bounds.
First, suppose $|H_{\mathbf p}|\notin\{n,n-2\}$.
    By Lemma~\ref{lem:assume_1/2}, it suffices to show that the set of vectors $\{ \mathfrak f_i \; : \; i \in H_{\mathbf p}\}$ is linearly independent.
    Suppose $\sum_{i \in H_{\mathbf p}} c_i \mathfrak f_i = 0$ for some coefficients $c_i \in \mathbb R$.
    Evaluating this sum at a permutation $x \in \mathfrak S_n$, we have
    \begin{equation}
        \label{eqn:sumH}
        \sum_{i \in H_{\mathbf p}} c_i \phi(x^{-1}(i)) = 0.
    \end{equation}
    For a fixed $j \in H_{\mathbf p}$, let $U_{j} \subset \mathfrak S_n$ denote the set of permutations where the label $j$ is in position $1$, that is,
    \[
    U_{j} = \{x \in \mathfrak S_n \; : \; x^{-1}(j) = 1 \}.
    \]
    Observe that, for any $i \in H_{\mathbf p} \setminus \{j\}$ and $m \in \{2,\dots,n\}$, the number of $x \in U_{j}$ such that $x^{-1}(i) = m$ is $(n-2)!$. 
    Therefore, summing \eqref{eqn:sumH} over $U_{j}$ yields
    \begin{equation}
        \label{eqn:sumHSjk}
        c_j (n-1)! \phi(1) + \sum_{i \in H_{\mathbf p} \setminus \{j\}} c_i (n-2)! \sum_{m \in \{2,\dots,n\}} \phi(m) = 0.
    \end{equation}
    Using the identity $\cos(\theta) = -\cos(\pi - \theta)$, we find that
    \[
        \sum_{m \in \{2,\dots,n\}} \phi(m) = -\phi(1).
    \]
    Then, since $\cos(\pi/(2n))\neq 0$, we can divide \eqref{eqn:sumHSjk} by $(n-2)! \phi(1)$ to obtain
    \[
        (n-1)c_j - \sum_{i \in H_{\mathbf p} \setminus \{j\}} c_i = 0.
    \]
    Let $C = \sum_{i \in H_{\mathbf p}} c_i$. 
    We can rewrite the above as $n c_j - C = 0$, which means $c_j = C/n$ for all $j \in H_{\mathbf p}$.
    Write $c = c_j$ for any $j \in H_{\mathbf p}$.
    Then $C = |H_{\mathbf p}|c$, and we can furthermore deduce $(n - |H_{\mathbf p}|)c = 0$.
    
    Since $|H_{\mathbf p}|\notin\{n,n-2\}$, in particular $|H_{\mathbf p}|<n$
    meaning 
    $n - |H_{\mathbf p}| > 0$. 
    This immediately forces $c = 0$, which implies $\dim \ker \left (K-\lambda_K I \right ) \ge |H_{\mathbf p}|$.

    The case of $|H_{\mathbf p}| = n$ follows from the proof of \cite[Theorem 14]{W}, which shows that $\dim \ker \left (K-\lambda_K I \right ) \ge n-1$.
 
We may now assume that $|H_{\mathbf p}| =n-2$.
By Lemma~\ref{lem:g}, it remains to show that $\mathfrak g$ does not lie in the span of $\{\mathfrak f_i:i\in H_{\mathbf p}\}$.
Indeed, suppose that $x$ and $x'$ are two permutations with the same positions for all labels in $H_{\mathbf p}$, but with the labels $1$ and $n$ interchanged, and $x^{-1}(1)<x^{-1}(n)$. 
Then $\mathfrak f_i(x)=\mathfrak f_i(x')$ for each $i \in H_{\mathbf p}$.
However, $\mathfrak g(x')=-p_{1,n}\mathfrak g(x)/p_{n,1}\neq \mathfrak g(x)$.
Hence, $\mathfrak g$ does not lie in the span of $\{\mathfrak f_i:i\in H_{\mathbf p}\}$, which implies $\dim \ker \left (K-\lambda_K I \right ) \ge n-1$.

Now we prove the corresponding upper bounds.
Suppose first that $|H_{\mathbf p}| \in \{n-2,n\}$. 
Define the function $\mathcal F: \ker \left (K-\lambda_K I \right ) \to \mathbb R^{n-1}$ by
\[
 f \mapsto \left (\mathsf E_1f(z_{1,2}), \dots, \mathsf E_1f(z_{1,n})\right).
\]
We claim that $\mathcal F$ is injective. 
Indeed, suppose that $\mathcal F(f)=0$ for some non-zero $f\in \ker \left (K-\lambda_K I \right )$.
Note that Lemma~\ref{lem:tech} applies since $\mathbf p$ is regular and Corollary~\ref{cor:reg} implies that $m_{\mathbf p} = 1/2$. 
Let $(1,b_f)$ be the lexicographically minimal element of $S_f$.
Then clearly $b_f \in \{2,\dots,n\}$ and 
$\mathsf E_1 f(z_{1,b_f})\neq 0$, contradicting $\mathcal F(f)=0$.
Hence $\mathcal F$ is injective, and thus $\dim \ker \left (K-\lambda_K I \right )\le n-1$.

Suppose next that $|H_{\mathbf p}|\notin \{n,n-2\}$. 
Similar to above, we define the function $\mathcal G: \ker \left (K-\lambda_K I \right ) \to \mathbb R^{|H_{\mathbf p}|}$ by $f \mapsto (\mathsf E_1f(z_{1,j}))_{j \in H_{\mathbf p}}.$
We claim that $\mathcal G$ is injective. 
Indeed, suppose that $\mathcal G(f)=0$ for some non-zero $f\in \ker \left (K-\lambda_K I \right )$.
Let $(1,b_f)$ be the lexicographically minimal element of $S_f$, which exists by Lemma~\ref{lem:tech}.
Since $|H_{\mathbf p}|\notin \{n,n-2\}$, by Lemma~\ref{lem:tech} and Lemma~\ref{lem:interval}, we have $b_f<n$, and moreover $b_f\in H_{\mathbf p}$.
Hence $\mathsf E_1f(z_{1,b_f})\neq 0$, contradicting $\mathcal G(f)=0$.
Thus, $\mathcal G$ is injective, and therefore we find that $\dim \ker \left (K-\lambda_K I \right )\le |H_{\mathbf p}|$. 
\end{proof}

\end{document}